\documentclass{amsart}

\usepackage{fullpage}	

\usepackage{amsmath}
\usepackage{amssymb}
\usepackage{color}
\usepackage{mathrsfs}
\usepackage{url}
\usepackage{enumitem}
\usepackage{stmaryrd}
\usepackage[all]{xy}
\usepackage{verbatim}

\newtheorem{Thm}{Theorem}[section]

\newtheorem{Lem}[Thm]{Lemma}
\newtheorem{Prop}[Thm]{Proposition}

\theoremstyle{definition}
\newtheorem{Def}[Thm]{Definition}

\theoremstyle{remark}
\newtheorem{Rmk}[Thm]{Remark}

\DeclareMathOperator{\Char}{Char}

\newcommand{\defi}[1]{\textsf{#1}} 				

\newcommand{\FF}{\mathbb{F}}

\newcommand{\HH}{\mathbb{H}}

\newcommand{\QQ}{\mathbb{Q}}

\newcommand{\HMW}{H_{\text{MW}}}

\usepackage{color}

\usepackage[
	backref,
	pdfauthor={David Brown}, 
]{hyperref}

\DeclareMathOperator{\Spec}{Spec}

\author{Christopher Davis}
\address{University of California, Irvine, Dept of
Mathematics, Irvine, CA 92697}
\email{davis@math.uci.edu}
\author{David Zureick-Brown}
\address{Dept. of Mathematics and Computer science, Emory University,
Atlanta, GA 30322}
\email{dzb@mathcs.emory.edu}
\date{\today}

\begin{document}

\title{Integral Monsky-Washnitzer cohomology and the overconvergent de Rham-Witt complex}

\maketitle 

\begin{abstract}
In their paper which introduced Monsky-Washnitzer cohomology, Monsky and Washnitzer described conditions under which the definition can be adapted to give \emph{integral} cohomology groups.  It seems to be well-known among experts that their construction always gives well-defined integral cohomology groups, but this fact also does not appear to be explicitly written down anywhere.  In this paper, we prove that the integral Monsky-Washnitzer cohomology groups are well-defined, for any nonsingular affine variety over a perfect field of characteristic~$p$.  We then compare these cohomology groups with overconvergent de\thinspace Rham-Witt cohomology.  It was shown earlier that if the affine variety has small dimension relative to the characteristic of the ground field, then the cohomology groups are isomorphic.  We extend this result to show that for any nonsingular affine variety, regardless of dimension, we have an isomorphism between integral Monsky-Washnitzer cohomology and overconvergent de\thinspace Rham-Witt cohomology in degrees which are small relative to the characteristic.
\end{abstract}

\section{Introduction}

Monsky-Washnitzer cohomology is usually defined rationally, but cases in which it could be defined integrally were already given in Monsky and Washnitzer's original paper \cite[Remark 3, p.~205]{MW68}.  The integral definition for arbitrary nonsingular affine varieties (and in particular the fact that it is well-defined) does not seem to be written down formally anywhere, and so we describe it below.  

Overconvergent de Rham-Witt cohomology was studied in \cite{DavisLZ:Rham}.  It was shown there that, rationally, it agreed with Monsky-Washnitzer cohomology.  Conditions were given on the nonsingular affine variety $\overline{X}$ under which the integral cohomology groups were isomorphic \cite[Corollary~3.25(a)]{DavisLZ:Rham}.  In this paper, we extend the result to show that certain of these cohomology groups are always isomorphic, without any conditions on the nonsingular affine variety $\overline{X}$.  In particular, the integral cohomology groups $H^0$ and $H^1$ are always isomorphic.  

Our main results are the following.  The authors suspect that the first is already known to experts.  (See the end of this section for our notation and conventions.)
\begin{Thm}  \label{main theorem} Let $\overline{X} = \Spec \overline{A}$ denote a nonsingular affine variety over a perfect field $k$ of characteristic $p$.  Let $A^{\dag}$ denote the weak completion of a nonsingular lift to $W(k)$.  
\begin{enumerate}
\item \label{MW part} The integral Monsky-Washnitzer cohomology groups $H^i(\Omega^{\bullet}_{A^{\dag}/W(k)})$ are well-defined. More precisely, given two lifts $A$ and $A'$ of $\overline{A}$, there exists an isomorphism
\[
H^i(\Omega^{\bullet}_{A^{\dag}/W(k)}) \xrightarrow{\sim} H^i(\Omega^{\bullet}_{A'^{\dag}/W(k)}).
\]
\item \label{small degree part} We have an isomorphism
\[
H^i(\Omega^{\bullet}_{A^{\dag}/W(k)}) \xrightarrow{\sim} H^i(W^{\dag}\Omega^{\bullet}_{\overline{A}})
\]
for all $i < p$.
\end{enumerate}
\end{Thm}

\begin{proof}
  The proof of (\ref{MW part}) is given in \cite[Remark~3, p.~205]{MW68} in the special case that $\overline{X}$ is a complete transversal intersection.  For the proof in the general case, see Section~\ref{MW section} below.  The key pieces to the proof are the local result of Monsky and Washnitzer just referenced, a result of Meredith concerning sheaf properties of the Monsky-Washnitzer complex, and a \v{C}ech spectral sequence argument.


The proof of (\ref{small degree part}) is given in Section~\ref{OdRW section}.  Again the strategy of the proof is to prove the result locally, which we do using results from \cite{DavisLZ:Rham}, and then to deduce the result in general using a \v{C}ech spectral sequence argument.
\end{proof}

\subsection*{Acknowledgements} This paper answers a question asked by Liang Xiao at the first author's thesis defense.  We thank Liang for his interest and for recognizing the opportunity to extend the previous results, and also for a careful reading of an earlier draft.  The first author also thanks his earlier coauthors Andreas Langer and Thomas Zink for their support and encouragement.  Thanks also to Kiran Kedlaya for many useful conversations and suggestions.

\subsection*{Notation}
Let $k$ denote a perfect field of characteristic $p$ and let $W(k)$ denote the ring of $p$-typical Witt vectors with coefficients in $k$.  By \defi{variety} over $k$, we mean a separated and integral scheme of finite type over the (perfect but not necessarily algebraically closed) field $k$.  Throughout this paper, let $\overline{A}$ denote the coordinate ring of a nonsingular affine variety over $k$.    We write $\overline{X} = \Spec \overline{A}$.  

\section{Integral Monsky-Washnitzer cohomology} \label{MW section}

Let $\overline{X} = \Spec \overline{A}$ denote a nonsingular affine variety over the perfect field $k$ of characteristic~$p$.   We first reproduce Monsky and Washnitzer's construction of the Monsky-Washnitzer cohomology groups associated to $\overline{A}$, with the key difference that we never tensor with $\QQ$.  Monsky and Washnitzer showed in special cases that these \emph{integral} Monsky-Washnitzer cohomology groups are well-defined; see \cite[Remark~3, p.~205]{MW68}.  Here we combine Monsky and Washnitzer's argument with results of Meredith to prove that the integral Monsky-Washnitzer cohomology groups are well-defined in general, i.e., for any nonsingular affine variety $\overline{X}$, they do not depend on a choice of lift to characteristic~0.

Keep notation as in the previous paragraph.  By \cite [Theorem 6]{elkik:lifting} we can lift $\Spec \overline{A}$ (in many ways) to an affine variety which is smooth over $\Spec W(k)$. For such a lift $\Spec A$ we define $\widehat{A}$ to be the $p$-adic completion of $A$.  Following \cite[Section~2.2]{Kedlaya:computingZetaFunctions}, we define the \defi{weak completion} (or alternatively the \defi{integral dagger algebra}) $A^{\dag}$ as the smallest $p$-adically saturated subring of $\widehat{A}$ containing $A$ and all series of the form 
\[
\sum_{i_1, \ldots, i_n \geq 0} c_{i_1, \ldots, i_n} x_1^{i_1} \cdots x_n^{i_n},
\] where $c_{i_1, \ldots, i_n} \in W(k)$ and $x_{j} \in pA^{\dag}$. 
When $A = W(k)[x_1, \ldots, x_n]$, $A^{\dag}$ is the set of series which converge on a $p$-adic polydisc of radius strictly greater than one, and in general,
by \cite[Theorem~2.2]{MW68}, we have the following alternative description of $A^{\dag}$: any surjection $k[x_1, \ldots, x_n] \to \overline{A}$ lifts to a map $W(k)\langle x_1, \ldots, x_n\rangle \to \widehat{A}$, and $A^{\dag}$ is the image of $W(k)\langle x_1, \ldots, x_n \rangle^{\dag}$ (this is independent of the choice of surjection). 

\begin{Prop} \label{maps lift}
  Let $\overline{\varphi}\colon \overline{A} \rightarrow \overline{B}$ denote a ring homomorphism.  Choose $A,\,B$, two lifts to characteristic~0 as above, and let $A^{\dag},\,B^{\dag}$ denote the corresponding weak completions.  Then there exists a (typically not unique) $p$-adically continuous ring homomorphism $\varphi\colon A^{\dag} \rightarrow B^{\dag}$ lifting $\varphi$.    
\end{Prop}

\begin{proof}
First note that $A^{\dag}$ is flat over $A$ by \cite[Proposition~1.3]{Mer71}. The rest of the assertion is proven in \cite[Theorem~2.4.4]{vdP86}.
\end{proof}

We define the \defi{module of continuous relative differentials} and the corresponding \defi{Monsky-Washnitzer complex} as follows.  First of all, if $A = W(k)[\overline{x}] := W(k)[x_1, \ldots, x_n]$, then $\Omega^1_{A^{\dag}}$ is the free $A^{\dag}$-module generated by $dx_1, \ldots, dx_n$.  The Monsky-Washnitzer complex
$\Omega^{\bullet}_{W(k)\langle\overline{x}\rangle^{\dag}} = \Omega^{\bullet}_{W(k)\langle\overline{x}\rangle^{\dag}/W(k)}$ is determined by the map 
\[
d\colon \sum_I c_I x^I \mapsto \sum_I \sum_{j = 1}^n i_jc_I \left(\frac{x^I}{x_j}\right)dx_j.
\]
In general, given $A \cong W(k)[\overline{x}]/\mathfrak{a}$,  $\Omega^{1}_{A^{\dag}} := \Omega^{1}_{A^{\dag}/W(k)}$ is the quotient of the $A^{\dag}$-module 
\[
A^{\dag} \otimes_{W(k)\langle\overline{x}\rangle^{\dag}}\Omega^{1}_{W(k)\langle\overline{x}\rangle^{\dag}/W(k)}
\] by the submodule generated by $dr$ for $r \in \mathfrak{a}$. Define $\Omega^i_{A^{\dag}} := \Lambda^i_{A^{\dag}} \Omega^1_{A^{\dag}}.$  Define a complex $\Omega^{\bullet}_{A^{\dag}}$ in the obvious way.  Note that this is not the same as the usual de\thinspace Rham complex over the ring $A^{\dag}$, because we have additional \emph{continuity} conditions such as 
\[
d\left(\sum p^i x^i\right) = \left( \sum ip^i x^{i-1} \right) dx.
\]

We now come to the definition of the integral Monsky-Washnitzer cohomology groups.  The main result of this section is that they are independent of choice of nonsingular lift from $\Spec \overline{A}$ to $\Spec A$.  

\begin{Def} 
Let $\Spec \overline{A}$ denote a nonsingular affine variety over $k$, and let $A^{\dag},\, \Omega^{\bullet}_{A^{\dag}/W(k)}$ be as above.  Define the \defi{(integral) Monsky-Washnitzer cohomology groups} to be $\HMW^i(\Spec \overline{A}) := H^i(\Omega^{\bullet}_{A^{\dag}/W(k)}).$  
\end{Def}

\begin{Def}
  Following \cite[Definition~3.3]{MW68}, we say that $\Spec \overline{B} \to  \Spec \overline{A}$ is a complete transversal intersection if $\overline{B} \cong \overline{A}[x_1,\ldots,x_n]/(\overline{F}_1,\ldots,\overline{F}_s)$, where the $s \times s$ subdeterminants of the matrix $\left(\overline{F}_{i,x_j}\right)$ generate the unit ideal in $\overline{B}$.
\end{Def}

\begin{Prop} \label{local well-definedness}
Let $\Spec \overline{A}$ denote a complete transversal intersection.  Let $\Spec A,\,\Spec A'$ denote two nonsingular lifts of $\Spec \overline{A}$, and let $A^{\dag},\,A'^{\dag}$ denote the corresponding dagger algebras.  Let $\varphi\colon A^{\dag} \rightarrow A'^{\dag}$ denote a map lifting the identity on $\overline{A}$, as in Proposition~\ref{maps lift}.  Then the induced map
\[
\varphi\colon \Omega^{\bullet}_{A^{\dag}/W(k)} \rightarrow \Omega^{\bullet}_{A'^{\dag}/W(k)}
\]
is a quasi-isomorphism.
\end{Prop}

\begin{proof}
See \cite[Remark 3, p.~205]{MW68}.  
\end{proof}

Our goal is to prove well-definedness of integral Monsky-Washnitzer cohomology.  Proposition~\ref{local well-definedness} proved the result in a special case.  We will check now that this can be reinterpreted as a local result.  In order to ``glue'', we will need certain sheaf properties which are provided by the work of Meredith.

\begin{Lem} \label{covering by complete transversal intersections}
Let $\overline{X}$ denote a nonsingular affine variety over $k$.  Then $\overline{X}$ can be covered by finitely many complete transversal intersections in such a way that all finite intersections are also complete transversal intersections.
\end{Lem}

\begin{proof}
We can cover $\overline{X}$ by complete transversal intersections by the Jacobian criterion.  It is well-known (often under the name Nike's Lemma) that we can cover $\overline{X}$ by open affines which are distinguished opens in both $\overline{X}$ and in the complete transversal intersections.  The result now follows from the fact that a distinguished open within a complete transversal intersection is itself a complete transversal intersection.
\end{proof}

\begin{Lem} \label{finiteness of pieces}
Let $\Spec \overline{A}$ denote a nonsingular affine variety over $k$ and let $A^{\dag}$ be an associated dagger algebra as above.  Then $\Omega^i_{A^{\dag}/W(k)}$ is a finitely generated $A^{\dag}$-module for any $i$.  
\end{Lem}

\begin{proof}
See \cite[Theorem~4.5]{MW68}.
\end{proof}

\begin{Def} \label{MW sheaf definition}
Let $\mathscr{F}^i_{A^{\dag}}$ denote the presheaf on $\Spec \overline{A}$ associated to $\Omega^i_{A^{\dag}/W(k)}$ as in \cite[Definition~4]{Mer71}.
\end{Def}

\begin{Rmk}
Note that the presheaf $\mathscr{F}^i_{A^{\dag}}$ defined in Definition~\ref{MW sheaf definition} depends on $A^{\dag}$, not only on $\overline{A}$.
\end{Rmk}

\begin{Lem} \label{MW acyclic}
The presheaf $\mathscr{F}^i_{A^{\dag}}$ defined above is a sheaf.  For any $i$ and for any $j > 0$, the sheaf cohomology $H^j(\Spec \overline{A}, \mathscr{F}^i_{A^{\dag}}) = 0$.
\end{Lem}

\begin{proof}
This follows from Lemma~\ref{finiteness of pieces} and \cite[Theorem~14]{Mer71}.
\end{proof}

\begin{Lem} \label{hypercohomology interpretation}
For any nonsingular $\Spec \overline{A}$, we have an isomorphism between the cohomology groups of the Monsky-Washnitzer complex $H^i(\Omega^{\bullet}_{A^{\dag}/W(k)})$ and the hypercohomology groups $\HH^i(\Spec \overline{A}, \mathscr{F}^{\bullet}_{A^{\dag}})$.  
\end{Lem}

\begin{proof}
By Lemma~\ref{MW acyclic},  $H^j(\Spec \overline{A}, \mathscr{F}^i_{A^{\dag}}) = 0$; the result thus follows from the hypercohomology spectral sequence \cite[Application~5.7.10]{Weibel:introductionToHomological}.
\end{proof}

\begin{Lem} \label{spectral sequence rectangle}
Fix $n \geq 0$.  Let $f\colon E \to \widetilde{E}$ be a morphism of first quadrant spectral sequences such that the induced map $f\colon E_2^{p,q} \to  \widetilde{E}_2^{p,q}$ is an isomorphism for all $p$ and all $q \leq n$.  
Then the induced map $E_{\infty}^{p+q} \to \widetilde{E}_{\infty}^{p+q}$ is an isomorphism
for $p+q \leq n$.  
\end{Lem}

\begin{proof}
For each $i$ and $p+q \leq n$, the hypotheses and an inspection of the domain and codomain of the differentials imply that the map $E_i^{p,q} \to  \widetilde{E}_i^{p,q}$ is an isomorphism.
\end{proof}

\begin{Rmk}
The condition in Lemma~\ref{spectral sequence rectangle} concerning bounds on the index $q$ will not be needed until Section~\ref{OdRW section}.
\end{Rmk}

To deduce a global result from our local results, we will need the following \v{C}ech spectral sequence for computing the hypercohomology of a complex of sheaves.  

\begin{Prop} \label{Cech spectral sequence}
Let $X$ denote a topological space, let $\mathscr{U}$ denote an open covering of $X$, and let $\mathscr{F}^{\bullet}$ denote a complex of sheaves.  Let $\mathscr{H}^{q}$ denote the presheaf on $X$ for which $\mathscr{H}^q(U) := \mathbb{H}^{q}(U, \mathscr{F}^{\bullet})$.  Then there is a first quadrant spectral sequence with
\[
E_2^{p,q} := \check{H}^p \left( \mathscr{U}, \mathscr{H}^q \right)
\]
and converging to $\HH^{p+q}(X,\mathscr{F}^{\bullet}).$
\end{Prop}

\begin{proof}
See \cite[\href{http://stacks.math.columbia.edu/tag/01GY}{Tag 01GY}]{stacks-project}.
\end{proof}

We now come to the proof that integral Monsky-Washnitzer cohomology is well-defined.
\begin{proof}[Proof of Theorem~\ref{main theorem}(\ref{MW part})]
By Proposition~\ref{maps lift}, there exists a map $\varphi\colon A^{\dag} \rightarrow A'^{\dag}$ lifting the identity on $\overline{A}$.  Fix a covering $\mathscr{U}$ of $\Spec \overline{A}$ as in Lemma~\ref{covering by complete transversal intersections}.  Let $E$ (resp., $E'$) denote the spectral sequence in Proposition~\ref{Cech spectral sequence} corresponding to the covering $\mathscr{U}$ and the complex of sheaves $\mathscr{F}^{\bullet}_{A^{\dag}}$ (resp., $\mathscr{F}^{\bullet}_{A'^{\dag}}$).  By Proposition~\ref{local well-definedness}, the map $\varphi$ induces a map of spectral sequences and isomorphisms $E_2^{p,q} \xrightarrow{\sim}  {E'}_2^{p,q}$ for all $p,q$.  Then by Lemma~\ref{spectral sequence rectangle} the map $\varphi$ induces an isomorphism $E_{\infty}^{p+q} \xrightarrow{\sim}  {E'}_{\infty}^{p+q}$.
Hence the induced map
\[
\varphi \colon \HH^{p+q}\left(X,\mathscr{F}_{A^{\dag}}^{\bullet}\right) \to \HH^{p+q}\left(X,\mathscr{F}_{A'^{\dag}}^{\bullet}\right)
\]
is an isomorphism.  We are then finished by Lemma~\ref{hypercohomology interpretation}.
\end{proof}

\section{Comparison with overconvergent de\thinspace Rham-Witt cohomology} \label{OdRW section}

In this section, we prove Theorem~\ref{main theorem}(\ref{small degree part}).  As in Section~\ref{MW section}, let $\overline{X} = \Spec \overline{A}$ denote a nonsingular affine variety over a perfect field $k$ of characteristic $p$.  
We will first define a (non-canonical) comparison morphism between the integral Monsky-Washnitzer complex $\Omega^{\bullet}_{A^{\dag}/W(k)}$ and the overconvergent de\thinspace Rham-Witt complex $W^{\dag} \Omega^{\bullet}_{\overline{A}}$.  
We will then find a covering of $\overline{X}$ over which this comparison morphism induces an isomorphism of certain cohomology groups. Lastly we will use Proposition~\ref{Cech spectral sequence} and Lemma~\ref{spectral sequence rectangle} to deduce the result for $\overline{X}$.  

Let $\Spec \overline{A}$ denote a nonsingular affine variety, and let $A^{\dag}$ denote a weakly complete lift of $\overline{A}$, as in Section~\ref{MW section}.  By Proposition~\ref{maps lift}, there exists a lift of Frobenius $F\colon A^{\dag} \rightarrow A^{\dag}$.  

\begin{Prop} \label{comparison functoriality}
Let $\overline{A},\,A^{\dag},\,F$ be as above.  There exists a unique ring homomorphism 
\[
s_F\colon A^{\dag} \rightarrow W(A^{\dag})
\]
such that, for any $a \in A^{\dag}$, the ghost components of $s_F(a)$ are $(a, F(a), F^2(a), \ldots)$.  This map is functorial in the sense that if we have $\overline{A}',\,A'^{\dag},\,F'$ as above and a map $\varphi$ such that the left-hand square in the following diagram commutes, then the right-hand square also commutes:
\[
\xymatrix{
A^{\dag} \ar[r]^{F} \ar[d]^{\varphi} & A^{\dag} \ar[d]^{\varphi} \ar[r]^{s_{F}} & W(A^{\dag}) \ar[d]^{W(\varphi)}\\
A'^{\dag} \ar[r]^{F'} & A'^{\dag} \ar[r]^{s_{F'}} & W(A'^{\dag}).
}
\]
\end{Prop}

\begin{proof}
See \cite[(0.1.3.16)]{Illusie:deRhamWitt}.
\end{proof}

\begin{Def} \label{definition of comparison map}
Keep notation as above.
Let $W^{\dag}\Omega_{\overline{A}}^{\bullet}$ denote the overconvergent de\thinspace Rham-Witt complex, as in \cite[Definition~1.1]{DavisLZ:Rham} and let $W^{\dag} \Omega^{\bullet}_{\overline{X}}$ denote the associated complex of Zariski sheaves on $X$ as in \cite[Corollary~1.6]{DavisLZ:Rham} .  Write $t_F$ for the composition $A^{\dag} \stackrel{s_F}{\rightarrow} W(A^{\dag}) \rightarrow W^{\dag}(\overline{A}),$ as in \cite[Proposition~3.2]{DavisLZ:Rham}.  Also write $t_F$ for the induced map $\Omega^{\bullet}_{A^{\dag}/W(k)} \rightarrow W^{\dag}\Omega^{\bullet}_{\overline{A}}$.    
\end{Def}

\begin{Prop} \label{t_F for sheaves}
Let $\overline{A},\,A^{\dag},\,F$ be as above and let $\mathscr{F}^{\bullet}_{A^{\dag}}$ be as in Definition~\ref{MW sheaf definition}.  Then the map $t_F$ defined in Definition~\ref{definition of comparison map} induces a map on complexes of Zariski sheaves
\[
t_F\colon \mathscr{F}^{\bullet}_{A^{\dag}} \rightarrow W^{\dag} \Omega^{\bullet}_{\overline{X}}.
\]
\end{Prop}

\begin{proof}
Because the maps on complexes are determined by the maps in degree zero, it suffices to prove this in degree zero. For any $f \in A^{\dag}$, let $A^{\dag}_{[f]}$ denote the weak completion of the localization $A^{\dag}_{f}$, as in \cite[Definition~2.1]{Mer71}.  The $p$-adically continuous ring homomorphism $F\colon A^{\dag} \rightarrow A^{\dag}$ extends uniquely to a $p$-adically continuous ring homomorphism $A^{\dag}_{[f]} \rightarrow A^{\dag}_{[f]}$.  Hence we are finished by the functoriality assertion in Proposition~\ref{comparison functoriality}.
\end{proof}

Our goal is to show that the comparison map $t_F$ defined above induces an isomorphism between (certain) integral Monsky-Washnitzer cohomology groups and (certain) overconvergent de\thinspace Rham-Witt cohomology groups.  This result will generalize \cite[Corollary~3.25(a)]{DavisLZ:Rham}, which states that if the dimension of $\Spec \overline{A}$ is small relative to the characteristic~$p$, then these comparison maps induce isomorphisms between (all) integral Monsky-Washnitzer cohomology groups and (all) overconvergent de\thinspace Rham-Witt cohomology groups. 

Our strategy for showing that these comparison maps induce isomorphisms in certain degrees will be to prove the result locally, and then deduce the result in general using the \v{C}ech spectral sequence.  The \v{C}ech covering of $\overline{X}$ we use will typically be finer than the covering used in Section~\ref{MW section}.

\begin{Def}
An affine variety is called \defi{special} if it is a distinguished open inside a scheme which is finite \'etale over affine space.
\end{Def}

\begin{Prop} \label{special affine covering}
Let $\overline{X}$ denote a nonsingular affine variety over $k$.  Then $\overline{X}$ can be covered by finitely many special affines in such a way that all finite intersections are also special.
\end{Prop}

\begin{proof}
The fact that $\overline{X}$ can be covered by special affines comes from \cite{Ked05}.  
Note that a distinguished open inside of a special affine is also a special affine.  We are now finished as in the proof of Lemma~\ref{covering by complete transversal intersections}.
\end{proof}

\begin{Lem} \label{covering for OdRW}
Let $\overline{X}$ denote a nonsingular affine variety over $k$.  Then $\overline{X}$ can be covered by finitely many open affines which are both special and complete transversal intersections, and in such a way that any finite intersection is also both special and a complete transversal intersection.  
\end{Lem}

\begin{proof}
Begin by finding two separate coverings, one as in Proposition~\ref{special affine covering} and one as in Lemma~\ref{covering by complete transversal intersections}.  These are both coverings by distinguished opens; in one case the distinguished opens are special affines and in the other case they are complete transversal intersections.  Because distinguished opens of special affines (resp., complete transversal intersections) are special affines (resp., complete transversal intersections), intersecting these two coverings completes the proof.
\end{proof}

The following proposition asserts that integral Monsky-Washnitzer cohomology is isomorphic to overconvergent de\thinspace Rham-Witt cohomology in the particular case that $\overline{X} = \Spec \overline{A}$ is a special affine.  It makes use of a certain comparison map $\sigma$ defined in \cite{DavisLZ:Rham}; the definition depends on a choice of presentation of $\overline{A}$.

\begin{Prop} \label{local quasi-isomorphism dRW}
Let $\overline{X} = \Spec \overline{A}$ denote a special affine.  Fix both a weakly complete lift $A^{\dag}$ and a comparison map
\[
\sigma\colon \Omega^{\bullet}_{A^{\dag}/W(k)} \rightarrow W^{\dag}\Omega^{\bullet}_{\overline{A}}
\] 
as in \cite[(3.5)]{DavisLZ:Rham}.  For all $i$, the induced map
\[
\sigma\colon H^i(\Omega^{\bullet}_{A^{\dag}}) \rightarrow H^i(W^{\dag} \Omega^{\bullet}_{\overline{A}})
\]
is an isomorphism.
\end{Prop}

\begin{proof}
See \cite[Theorem~3.19]{DavisLZ:Rham}.
\end{proof}

Note that Proposition~\ref{local quasi-isomorphism dRW} asserts an isomorphism between the cohomology groups in all degrees, but it requires a particular choice of comparison map $\sigma$.  For $\overline{X}$ arbitrary, not necessarily special affine, there is not an obvious analogue of $\sigma$.  For this reason, below we will take for our comparison map $t_F$, as in Definition~\ref{definition of comparison map}.  Upon restricting to a special affine, we are only able to prove that the maps induced on cohomology by $\sigma$ and $t_F$ agree in suitably small degrees.  This explains the restriction on degree in Theorem~\ref{main theorem}(\ref{small degree part}).

\begin{Prop} \label{independent of F}
Let $\overline{X} = \Spec \overline{A}$ denote a special affine.   Fix a weakly complete lift $A^{\dag}$ as in Section~\ref{MW section}, and fix a lift of Frobenius $F\colon A^{\dag} \rightarrow A^{\dag}$.  Let
\[
t_F\colon \Omega^{\bullet}_{A^{\dag}/W(k)} \rightarrow W^{\dag} \Omega^{\bullet}_{\overline{A}}
\]
be the comparison morphism from Definition~\ref{definition of comparison map}.
Then for all $i < p$, the induced map
\[
t_F\colon H^i(\Omega^{\bullet}_{A^{\dag}/W(k)}) \rightarrow H^i(W^{\dag} \Omega^{\bullet}_{\overline{A}})
\]
is the same as the isomorphism $\sigma$ from Proposition~\ref{local quasi-isomorphism dRW}.  In particular, the maps on cohomology induced by $t_F$ are independent of the choice of Frobenius lift $F$.
\end{Prop}

\begin{proof}
We deduce this from the proof of \cite[Proposition~3.27]{DavisLZ:Rham} as follows.  Let 
\[
\phi_1,\, \phi_2\colon \Omega^{\bullet}_{A^{\dag}/W(k)} \rightarrow W^{\dag} \Omega^{\bullet}_{\overline{A}}
\]
denote two morphisms of differential graded algebras which induce the identity map on $\overline{A}$; we will apply this to the two maps $\sigma$ and $t_F$.   Choose $i < p$ and let $\omega \in \Omega^{i}_{A^{\dag}/W(k)}$ denote a cocycle.  We would like to show that $\phi_1(\omega) - \phi_2(\omega)$ is a coboundary in $W^{\dag} \Omega^{i}_{\overline{A}}$.  

We claim that for $i < p$ and for the maps $\varphi, L$ defined in the proof of \cite[Proposition~3.27]{DavisLZ:Rham}, the map
\[
L \circ \varphi\colon \Omega^{i}_{A^{\dag}/W(k)} \rightarrow W^{\dag} \Omega^{i-1}_{\overline{A}}
\]
is such that 
\[
d\left( L \circ \varphi(\omega) \right) =  \phi_1(\omega) - \phi_2(\omega).
\]
By the proof of \cite[Proposition~3.27]{DavisLZ:Rham}, we need only to check that the map $L \circ \varphi$ does indeed have image in $W^{\dag} \Omega^{i-1}_{\overline{A}}$.  (The proof in \cite{DavisLZ:Rham} shows this only after multiplying by $p^{\kappa}$, for $\kappa = \lfloor \log_p \dim \overline{X} \rfloor$.)  By \cite[Lemma~3.31(i)]{DavisLZ:Rham} and the definition of $L$, we are reduced to proving that if $\omega_j \in W^{\dag} \Omega^{i-1}_{\overline{A}}$ is divisible by $p^{\max(j-i+1,0)}$ for some $j \geq 0$, then 
\[
\frac{\omega_j}{j+1} \in W^{\dag} \Omega^{i-1}_{\overline{A}}.
\]
If $\max(j-i+1,0) = 0$, then $j + 1 \leq i < p$, and the result is trivial, since $j+1$ is a unit in $W(k)$.  So it remains to consider the case $\max(j-i+1,0) = j-i+1 > 0.$  The first case in which this occurs is $j = i$.  Recalling that $i < p$, we are finished in this case because $i - i + 1 = 1 \geq \log_p(i+1)$.  The remaining cases are also easy, because $j - i + 1$ grows faster with $j$ than $\log_p(j+1)$.  
\end{proof}

We are now ready to prove the second half of our main theorem.

\begin{proof}[Proof of Theorem~\ref{main theorem}(\ref{small degree part})]
Fix a weakly complete lift $A^{\dag}$ of $A$, and a lift of Frobenius $F\colon A^{\dag} \rightarrow A^{\dag}$.  Let $\mathscr{F}^{\bullet}_{A^{\dag}}$ denote the complex of sheaves associated to $A^{\dag}$ as in Definition~\ref{MW sheaf definition}. Let 
\[
t_F\colon \mathscr{F}^{\bullet}_{A^{\dag}/W(k)} \rightarrow W^{\dag} \Omega^{\bullet}_{\overline{A}}
\]
denote the morphism of sheaves, as in Proposition~\ref{t_F for sheaves}.  Now choose a distinguished open covering $\mathscr{U}$ of $\Spec \overline{A}$ as in Lemma~\ref{covering for OdRW}.   Let $E$ (resp., $E'$) denote the spectral sequence in Proposition~\ref{Cech spectral sequence} corresponding to the covering $\mathscr{U}$ and the complex of sheaves $\mathscr{F}^{\bullet}_{A^{\dag}}$ (resp., $W^{\dag}\Omega^{\bullet}_{\overline{X}}$).  Let $U_{I}$ denote an arbitrary finite intersection of opens $U_{\alpha} \in \mathscr{U}$ and consider the induced map on cohomology 
\[
t_F\colon \HH^i(U_I, \mathscr{F}_{A^{\dag}}^{\bullet}) \rightarrow \HH^i(U_I, W^{\dag} \Omega^{\bullet}_{\overline{A}}).
\]
By Lemma~\ref{MW acyclic} (resp., \cite[Proposition~1.2(b)]{DavisLZ:Rham}) the individual sheaves in the complex $\mathscr{F}_{A^{\dag}}^{\bullet}$ (resp., $W^{\dag}\Omega^{\bullet}_{\overline{X}}$) have trivial sheaf cohomology.  
Hence by Lemma~\ref{hypercohomology interpretation} and Proposition~\ref{independent of F}, the maps $t_F$ above are isomorphisms for $i < p = \Char \overline{A}$.  (Our reason to avoid the letter $p$ will be clear in the next paragraph.)

Thus $t_F$ induces a map of spectral sequences and isomorphisms $E_2^{p,q} \xrightarrow{\sim}  {E'}_2^{p,q}$ for all $p$ and all $q < \Char \overline{A}$.  Then by Lemma~\ref{spectral sequence rectangle} the map $t_F$ induces an isomorphism $E_{\infty}^{p+q} \xrightarrow{\sim}  {E'}_{\infty}^{p+q}$ for $p + q < \Char \overline{A}$.
Hence the induced map
\[
t_F \colon \HH^{p+q}\left(X,\mathscr{F}_{A^{\dag}}^{\bullet}\right) \to \HH^{p+q}\left(X,W^{\dag}\Omega^{\bullet}_{\overline{X}}\right)
\]
is an isomorphism for $p + q < \Char \overline{A}$.  We are finished by Lemma~\ref{hypercohomology interpretation}.
\end{proof}

\begin{Rmk}
It would be interesting to have an example in which there is not an integral isomorphism between Monsky-Washnitzer cohomology and overconvergent de\thinspace Rham-Witt cohomology (or to prove that no such example exists).  By Theorem~\ref{main theorem}, the dimension of $\Spec \overline{A}$ in such an example would have to be at least the characteristic~$p$.  Perhaps the first class to consider when seeking such an example would be the two-dimensional affine varieties over $\FF_{2}$.  
\end{Rmk}


\bibliography{Integral}

\begin{thebibliography}{10}

\bibitem{DavisLZ:Rham}
Christopher Davis, Andreas Langer, and Thomas Zink.
\newblock Overconvergent de {R}ham-{W}itt cohomology.
\newblock {\em Ann. Sci. \'Ec. Norm. Sup\'er. (4)}, 44(2):197--262, 2011.

\bibitem{elkik:lifting}
Ren{\'e}e Elkik.
\newblock Solutions d'\'equations \`a coefficients dans un anneau hens\'elien.
\newblock {\em Ann. Sci. \'Ecole Norm. Sup. (4)}, 6:553--603 (1974), 1973.

\bibitem{Illusie:deRhamWitt}
Luc Illusie.
\newblock Complexe de de\thinspace {R}ham-{W}itt et cohomologie cristalline.
\newblock {\em Ann. Sci. \'Ecole Norm. Sup. (4)}, 12(4):501--661, 1979.

\bibitem{Kedlaya:computingZetaFunctions}
Kiran~S. Kedlaya.
\newblock Computing zeta functions via {$p$}-adic cohomology.
\newblock In {\em Algorithmic number theory}, volume 3076 of {\em Lecture Notes
  in Comput. Sci.}, pages 1--17. Springer, Berlin, 2004.

\bibitem{Ked05}
Kiran~S. Kedlaya.
\newblock More \'etale covers of affine spaces in positive characteristic.
\newblock {\em J. Algebraic Geom.}, 14(1):187--192, 2005.

\bibitem{Mer71}
David Meredith.
\newblock Weak formal schemes.
\newblock {\em Nagoya Math. J.}, 45:1--38, 1972.

\bibitem{MW68}
P.~Monsky and G.~Washnitzer.
\newblock Formal cohomology. {I}.
\newblock {\em Ann. of Math. (2)}, 88:181--217, 1968.

\bibitem{stacks-project}
The {Stacks Project Authors}.
\newblock {\itshape Stacks Project}.
\newblock \url{http://stacks.math.columbia.edu}.

\bibitem{vdP86}
Marius van~der Put.
\newblock The cohomology of {M}onsky and {W}ashnitzer.
\newblock {\em M\'em. Soc. Math. France (N.S.)}, (23):4, 33--59, 1986.
\newblock Introductions aux cohomologies $p$-adiques (Luminy, 1984).

\bibitem{Weibel:introductionToHomological}
Charles~A. Weibel.
\newblock {\em An introduction to homological algebra}, volume~38 of {\em
  Cambridge Studies in Advanced Mathematics}.
\newblock Cambridge University Press, Cambridge, 1994.

\end{thebibliography}
\bibliographystyle{plain}

\end{document}